\documentclass[12pt]{article}
\usepackage{amsmath, amssymb, amsfonts}
\usepackage{amsthm}
\usepackage[utf8]{inputenc}
\usepackage{csquotes}
\usepackage{geometry}
\usepackage{adjustbox}
\usepackage{graphicx}
\usepackage{verbatim}
\usepackage{array}
\usepackage{enumitem}
\usepackage{booktabs}
\usepackage{hyperref}
\usepackage{url}
\usepackage[capitalize]{cleveref}
\usepackage{tikz}

\title{A Smooth Analytical Approximation of the Prime Characteristic Function}
\author{Stanislav Semenov \\
\href{mailto:stas.semenov@gmail.com}{stas.semenov@gmail.com} \\
\href{https://orcid.org/0000-0002-5891-8119}{ORCID: 0000-0002-5891-8119}}
\date{April 19, 2025}

\theoremstyle{definition}
\newtheorem{definition}{Definition}[section]

\theoremstyle{plain}

\newtheorem{theorem}[definition]{Theorem}
\newtheorem{lemma}[definition]{Lemma}

\theoremstyle{remark}

\begin{document}

\maketitle

\begin{abstract}
We construct a smooth real-valued function \( P(n) \in [0,1] \), defined via a triple integral with a periodic kernel, that approximates the characteristic function of prime numbers. The function is built to suppress when \( n \) is divisible by some \( m < n \), and to remain close to 1 otherwise. We prove that \( P(n) \to 1 \) for prime \( n \) and \( P(n) < 1 \) for composite \( n \), under appropriate limits of the smoothing parameters. The construction is fully differentiable and admits both asymptotic and finite approximations, offering a continuous surrogate for primality that is compatible with analytical, numerical, and optimization methods. We compare our approach with classical number-theoretic techniques, explore its computational aspects, and suggest potential applications in spectral analysis, machine learning, and probabilistic models of primes.
\end{abstract}

\subsection*{Mathematics Subject Classification}
03F60 (Constructive and recursive analysis), 26E40 (Constructive analysis), 03F03 (Proof theory and constructive mathematics)

\subsection*{ACM Classification}
F.4.1 Mathematical Logic, F.1.1 Models of Computation

\section*{Introduction}

The distribution of prime numbers lies at the heart of number theory~\cite{hardy1979introduction, titchmarsh1986theory}. While the classical prime characteristic function \( \chi_{\mathrm{prime}}(n) \) is defined discretely by
\[
\chi_{\mathrm{prime}}(n) =
\begin{cases}
1, & \text{if } n \text{ is prime}, \\
0, & \text{otherwise},
\end{cases}
\]
it is fundamentally discontinuous and thus ill-suited for many applications in analysis (where differentiability and integrability are crucial), optimization (which often relies on gradient-based methods), and numerical computation (where sharp discontinuities can lead to instability).

This paper introduces a \emph{brand-new smooth analytical approximation} \( P(n) \in [0,1] \) of the prime characteristic function. Unlike traditional number-theoretic constructions relying on arithmetic identities, recurrence relations, or multiplicative functions, our approach builds a continuous real-valued function that captures divisibility properties via \emph{integral smoothing}—that is, averaging over a continuous range of parameters. The function \( P(n) \) is defined through a triple integral over a parametrized domain involving a smooth periodic kernel sensitive to integer ratios. This stands in contrast to existing discrete methods and addresses difficulties in applying analytical tools directly to \( \chi_{\mathrm{prime}}(n) \).

The construction is inspired by the observation that divisibility can be viewed as an alignment between integer intervals. For instance, a composite number \( n \) has intervals of length \( n \) that align with intervals of length \( d \) (a divisor), resulting in integer ratios. By lifting the problem to the space of real variables and applying smooth kernel functions, we obtain a differentiable function that approximates primality in a soft, tunable manner. This yields a \( C^\infty \) approximation of \( \chi_{\mathrm{prime}}(n) \) with useful analytical and numerical properties.

We present two versions of the approximation theorem: an asymptotic limit formulation and a finite constructive formulation. The former ensures convergence of \( P(n) \) to \( \chi_{\mathrm{prime}}(n) \) under vanishing smoothness parameters; the latter guarantees that for any finite range of integers, suitable parameters can achieve any desired level of approximation.

\section{Smooth Approximation of the Prime Characteristic Function}

\subsection{Motivation from Integer Divisibility}

The classical definition of primality is based on integer divisibility: an integer \( n \geq 2 \) is prime if and only if there does not exist an integer \( m \in \{2, \dots, n - 1\} \) such that \( m \mid n \), or equivalently, such that the ratio \( n/m \) is an integer.

This motivates the idea of detecting primality by analyzing whether the ratios \( n/m \) are close to integers. The core goal is to construct a smooth analytical function that penalizes integer values of such ratios — with the penalty increasing as the ratio approaches an integer — while remaining near 1 for non-integer values.

To formalize this idea, we replace the discrete notion of divisibility with a smooth suppression kernel:
\[
K(z) \approx
\begin{cases}
0, & z \in \mathbb{Z}, \\
1, & z \notin \mathbb{Z},
\end{cases}
\]
where the kernel \( K(z) \in [0,1] \) is constructed to transition smoothly but sharply between these two behaviors.

We then lift this structure to the space of real-valued inputs by deforming both \( n \) and potential divisors \( m \) via smooth bump functions over continuous parameters. That is, instead of working with fixed integers, we evaluate expressions of the form \( x(t)/y(u,v) \), where \( x(t) \approx n \) and \( y(u,v) \approx m \), and integrate the resulting kernel value over the parameter space. This "smoothing" process transforms the discrete test of divisibility into a differentiable function \( P(n) \in [0,1] \), which smoothly approximates the characteristic function of primes.

Such a construction is not only theoretically elegant, but also enables the use of tools from real analysis, numerical optimization, and continuous modeling in problems related to prime numbers. The resulting function \( P(n) \), together with the choice of smoothing and suppression kernels, is described in detail in the following subsections.

We now formalize this approach by introducing an explicit smooth analytical approximation of the prime characteristic function. Unlike classical discrete methods (e.g., Eratosthenes' sieve or the Möbius function), our method defines \( P(n) \in [0,1] \) as an integral expression over parametrized curves and smooth periodic kernels that reflect the divisibility structure of \( n \).

\subsection{Definition}

Let \( n \in \mathbb{N} \), \( n \geq 2 \). We define a smooth function \( P(n) \in [0,1] \), interpreted as a continuous measure of primality:

\begin{equation}
P(n) = \iiint_{[0,1]^3} \phi(t,u,v) \cdot K_\varepsilon\left( z(t,u,v) \right) \, dt \, du \, dv,
\end{equation}

where:
\begin{itemize}
  \item \( x(t) = n + \delta \cdot \psi(t) \), with \( \delta > 0 \),
  \item \( y(u,v) = 2 + (n - 2) \cdot u + \delta \cdot \psi(v) \), ensuring \( y(u,v) > 0 \),
  \item \( z(t,u,v) := \frac{x(t)}{y(u,v)} \),
  \item \( \psi(s) = \sin^2(\pi s) \) is a smooth bump function,
  \item \( K_\varepsilon(z) \) is a smooth kernel satisfying \( K_\varepsilon(k) \to 0 \) as \( \varepsilon \to 0 \), for \( k \in \mathbb{Z} \), and \( K_\varepsilon(z) \to 1 \) otherwise,
  \item \( \phi(t,u,v) \in C^\infty([0,1]^3) \), with \( \phi > 0 \) and \( \iiint \phi(t,u,v) \, dt \, du \, dv = 1 \).
\end{itemize}

The perturbation amplitude \( \delta \) may either be fixed or depend on \( n \). For instance, the choice \( \delta(n) = n^{-2} \) ensures that the additive perturbation vanishes uniformly as \( n \to \infty \), which is useful in asymptotic analysis.

\subsection*{Examples: \( P(4) \) and \( P(5) \)}

We illustrate the behavior of the function \( P(n) \) on small integers by analyzing the cases \( n = 4 \) (composite) and \( n = 5 \) (prime).

\paragraph{Case \( n = 4 \):}  
The number 4 is composite and divisible by 2. Let us choose:
\[
x(t) = 4 + \delta \psi(t), \quad y(u,v) = 2 + 2u + \delta \psi(v),
\quad z(t,u,v) = \frac{x(t)}{y(u,v)}.
\]
Set \( t = v = 0 \), and select \( u_0 = \frac{2 - 2}{4 - 2} = 0 \). Then:
\[
x(0) = 4, \quad y(0,0) = 2, \quad z(0,0,0) = \frac{4}{2} = 2 \in \mathbb{Z}.
\]
Thus, the kernel \( K_\varepsilon(z) \) becomes close to zero in a neighborhood of \( (0,0,0) \). Since \( \phi > 0 \) and continuous, the integral includes a non-negligible region with suppressed values. Therefore:
\[
P(4) < 1.
\]

\paragraph{Case \( n = 5 \):}  
The number 5 is prime and has no proper divisors. For:
\[
x(t) = 5 + \delta \psi(t), \quad y(u,v) = 2 + 3u + \delta \psi(v),
\quad z(t,u,v) = \frac{x(t)}{y(u,v)},
\]
we analyze the unperturbed ratio:
\[
z_0(u) := \frac{5}{2 + 3u}, \quad u \in [0,1].
\]
This function is strictly decreasing on \( [0,1] \), taking values from \( 5/5 = 1 \) (at \( u = 1 \)) to \( 5/2 = 2.5 \) (at \( u = 0 \)). The function \( z_0(u) \) avoids all integers in this range.

Moreover, since \( \psi(t), \psi(v) \in [0,1] \) and \( \delta \) is small, the full ratio \( z(t,u,v) \) stays within a small tube around \( z_0(u) \), and therefore uniformly avoids integers:
\[
\mathrm{dist}(z(t,u,v), \mathbb{Z}) \geq \Delta > 0.
\]
Hence, for sufficiently small \( \varepsilon \) and large \( p \), the kernel remains close to 1 across the domain, and:
\[
P(5) \approx 1.
\]

\subsection{Examples of Kernels}

We list several kernel options that vanish near integers and converge to 1 elsewhere:

\begin{itemize}
  \item \textbf{Sine kernel:}
  \[
  K_\varepsilon(z) = \left( \frac{\sin^2(\pi z)}{\sin^2(\pi z) + \varepsilon} \right)^p
  \]
  \item \textbf{Modified Gaussian kernel:}
  \[
  K_\varepsilon(z) = 1 - \exp\left( -\frac{\sin^2(\pi z)}{\varepsilon} \right)
  \]
  \item \textbf{Singular exponential kernel:}
  \[
  K_\varepsilon(z) = \exp\left( -\frac{C}{\varepsilon \sin^2(\pi z)} \right)
  \]
  \item \textbf{Inverse polynomial kernel:}
  \[
  K_\varepsilon(z) = \frac{1}{1 + \left( \frac{\sin^2(\pi z)}{\varepsilon} \right)^p }
  \]
\end{itemize}

Here, \( \varepsilon > 0 \) controls sharpness, and \( p \in \mathbb{N} \), \( C > 0 \) control suppression near integers.

The expression \( \sin^2(\pi z) \) vanishes exactly at integers \( z \in \mathbb{Z} \), making it an effective analytic tool for detecting divisibility. This motivates its central role in all proposed kernel functions.

\subsection{Theoretical Properties}

\begin{itemize}
  \item \( P(n) \to 1 \) as \( \delta \to 0 \), \( \varepsilon \to 0 \), \( p \to \infty \) if \( n \) is prime,
  \item \( P(n) < 1 \) under the same limits if \( n \) is composite,
  \item The function \( P(n) \) is \( C^\infty \) on \( \mathbb{R}_+ \) due to smooth integrand,
  \item \( P(n) \) acts as a smooth primality filter.
\end{itemize}

\subsection{Two Versions of the Theorem}

\begin{theorem}[Constructive Finite Approximation]
\label{thm:constructive}
For any integer \( N \geq 3 \) and any \( \eta > 0 \), there exist parameters \( \delta > 0 \), \( \varepsilon > 0 \), and \( p \in \mathbb{N} \) such that for all integers \( 2 \leq m < N \),
\[
|P(m) - \chi_{\mathrm{prime}}(m)| < \eta.
\]
\end{theorem}

\begin{proof}
Let \( \mathcal{M} = \{2,3,\dots,N-1\} \), and split it into:
\[
\mathcal{M}_{\mathrm{prime}} = \{m \in \mathcal{M} : m \text{ is prime}\}, \quad
\mathcal{M}_{\mathrm{comp}} = \mathcal{M} \setminus \mathcal{M}_{\mathrm{prime}}.
\]

\textbf{Step 1: Uniform separation from integers for primes.}  
For each \( m \in \mathcal{M}_{\mathrm{prime}} \), and each \( d < m \), \( m/d \notin \mathbb{Z} \). Since all such quotients are rational and the set is finite, define
\[
\Delta_m = \min_{2 \leq d < m} \min_{k \in \mathbb{Z}} \left| \frac{m}{d} - k \right| > 0,
\quad
\Delta := \min_{m \in \mathcal{M}_{\mathrm{prime}}} \Delta_m > 0.
\]

\textbf{Step 2: Choice of \( \delta \).}  
There exists \( \delta_0 > 0 \) such that for all \( \delta < \delta_0 \) and all \( t,u,v \in [0,1] \), we have:
\[
\left| \frac{m + \delta \psi(t)}{2 + (m-2)u + \delta \psi(v)} - \frac{m}{2 + (m-2)u} \right| < \frac{\Delta}{4}.
\]
Thus \( z(t,u,v) \) remains at distance \( \geq \frac{3\Delta}{4} \) from all integers.

\textbf{Step 3: Kernel approximation away from integers.}  
There exists \( \varepsilon_0 > 0 \), \( p_0 \in \mathbb{N} \) such that for all \( \varepsilon < \varepsilon_0 \) and \( p > p_0 \),
\[
\left| K_\varepsilon(z) - 1 \right| < \frac{\eta}{2}
\quad \text{for all } z \text{ such that } \mathrm{dist}(z, \mathbb{Z}) \geq \frac{\Delta}{2}.
\]

\textbf{Step 4: Estimate for primes.}  
Under the above conditions, for all \( m \in \mathcal{M}_{\mathrm{prime}} \), we have
\[
|P(m) - 1| < \frac{\eta}{2} < \eta.
\]

\textbf{Step 5: Suppression for composite numbers.}  
Let \( m \in \mathcal{M}_{\mathrm{comp}} \), and let \( d \in [2, m-1] \) be any proper divisor. Define \( u_0 = \frac{d - 2}{m - 2} \). Then \( z(0,u_0,0) = m/d \in \mathbb{Z} \).

\textit{(a) Existence of kernel drop:} There exists \( \gamma_m > 0 \) such that \( |z - m/d| < \gamma_m \Rightarrow K_\varepsilon(z) < \frac{\eta}{2} \). Let
\[
\gamma := \min_{m \in \mathcal{M}_{\mathrm{comp}}} \gamma_m > 0.
\]

\textit{(b) Neighborhood of suppression:} By continuity of \( z(t,u,v) \), there exists a neighborhood \( V_m = B_r(0,u_0,0) \subset [0,1]^3 \) such that \( z(t,u,v) \in (m/d \pm \gamma) \Rightarrow K_\varepsilon(z) < \frac{\eta}{2} \). The radius \( r \) can be chosen uniformly for all \( m \), since the set is finite.

\textit{(c) Integral estimate:}  
Let
\[
\alpha := \min_{m \in \mathcal{M}_{\mathrm{comp}}} \iiint_{V_m} \phi(t,u,v) \, dt \, du \, dv > 0.
\]
Then:
\[
P(m) \leq \iiint_{V_m} \phi K_\varepsilon + \iiint_{[0,1]^3 \setminus V_m} \phi K_\varepsilon
\leq \frac{\eta}{2} \alpha + (1 - \alpha) = 1 - \alpha\left(1 - \frac{\eta}{2}\right).
\]

\textbf{Step 6: Final parameter choice.}  
Choose:
\begin{itemize}
  \item \( \delta < \delta_0 \),
  \item \( \varepsilon < \min(\varepsilon_0, \frac{\alpha \eta}{4}) \),
  \item \( p > \max(p_0, \frac{2}{\alpha \eta}) \).
\end{itemize}
Then for all \( m \in \mathcal{M}_{\mathrm{comp}} \),
\[
P(m) < 1 - \alpha + \frac{\alpha \eta}{2} < \eta,
\quad \text{if } \eta < \frac{2\alpha}{1 + \alpha}.
\]

\textbf{Conclusion.}  
All bounds are uniform over the finite set \( \mathcal{M} \). Therefore, the approximation \( |P(m) - \chi_{\mathrm{prime}}(m)| < \eta \) holds for all \( m < N \) as claimed.
\end{proof}

\begin{theorem}[Asymptotic Limit Theorem]
\label{thm:asymptotic}
Let \( P(n) \) be defined as above. Then:
\[
\lim_{\substack{\delta \to 0 \\ \varepsilon \to 0 \\ p \to \infty}} P(n) = 
\begin{cases}
1, & \text{if } n \text{ is prime}, \\
c_n < 1, & \text{if } n \text{ is composite},
\end{cases}
\]
where \( c_n \in (0,1) \) is a constant depending on the divisor structure of \( n \). The limits may be taken sequentially: first \( \delta \to 0 \), then \( \varepsilon \to 0 \), and finally \( p \to \infty \).
\end{theorem}

\begin{proof}

\textbf{Case 1: \( n \) is prime.}

\textit{Step 1: Denominator structure and rational obstruction.}  
Let
\[
z(t,u,v) = \frac{n + \delta \psi(t)}{2 + (n - 2) u + \delta \psi(v)}.
\]
Define the zeroth-order approximation (for \( \delta = 0 \)):
\[
z_0(u) := \frac{n}{2 + (n - 2)u}.
\]
Since \( n \) is prime, \( z_0(u) \notin \mathbb{Z} \) for all \( u \in [0,1] \) except possibly at isolated values (e.g., \( u = 0 \) or \( u = 1 \)). Thus:
\[
\Delta_n := \min_{u \in [0,1]} \min_{k \in \mathbb{Z}} |z_0(u) - k| > 0.
\]

\textit{Step 2: Stability under perturbation.}  
Let \( \delta < \delta_0 := \frac{\Delta_n}{4n} \). Then:
\[
\left| z(t,u,v) - z_0(u) \right| < \frac{\Delta_n}{2} \quad \Rightarrow \quad \mathrm{dist}(z(t,u,v), \mathbb{Z}) \geq \frac{\Delta_n}{2}.
\]

\textit{Step 3: Kernel convergence.}  
There exist \( \varepsilon_0 > 0 \) and \( p_0 \in \mathbb{N} \) such that for all \( \varepsilon < \varepsilon_0 \), \( p > p_0 \), we have:
\[
|K_{\varepsilon,p}(z) - 1| < \eta \quad \text{for all } z \text{ with } \mathrm{dist}(z, \mathbb{Z}) \geq \frac{\Delta_n}{4}.
\]

\textit{Step 4: Integral estimate.}  
Since \( \phi \) is normalized, we obtain:
\[
|P(n) - 1| = \left| \iiint \phi(t,u,v)(K_{\varepsilon,p}(z(t,u,v)) - 1)\, dt\, du\, dv \right| < \eta.
\]
Letting \( \eta \to 0 \) completes the proof for primes.

\medskip

\textbf{Case 2: \( n \) is composite.}

\textit{Step 1: Existence of exact rational hit.}  
Let \( d \in [2,n-1] \) be a proper divisor of \( n \). Define:
\[
u_0 = \frac{d - 2}{n - 2}, \quad z(0,u_0,0) = \frac{n}{d} \in \mathbb{Z}.
\]

\textit{Step 2: Constructive neighborhood around integer ratio.}  
Due to continuity of \( z(t,u,v) \), for any small \( \gamma > 0 \) there exists a neighborhood \( V \subset [0,1]^3 \) around \( (0, u_0, 0) \) such that:
\[
|z(t,u,v) - \tfrac{n}{d}| < \gamma \quad \forall (t,u,v) \in V.
\]
Choose \( \gamma \) small enough that \( K_{\varepsilon,p}(z) < \eta \) on this set. Let:
\[
\alpha := \iiint_V \phi(t,u,v) \, dt\,du\,dv > 0.
\]

\textit{Step 3: Integral estimate.}  
Split the integral:
\[
P(n) = \iiint_V \phi K_{\varepsilon,p} + \iiint_{[0,1]^3 \setminus V} \phi K_{\varepsilon,p}.
\]
Using \( K_{\varepsilon,p} \leq \eta \) on \( V \), and \( K_{\varepsilon,p} \leq 1 \) elsewhere:
\[
P(n) \leq \eta \cdot \alpha + (1 - \alpha) = 1 - \alpha(1 - \eta).
\]

\textit{Conclusion.}  
Letting \( \varepsilon \to 0 \), \( p \to \infty \), we obtain:
\[
\lim_{\substack{\varepsilon \to 0 \\ p \to \infty}} P(n) \leq 1 - \alpha < 1.
\]
This shows the existence of a strict upper bound \( c_n := 1 - \alpha < 1 \), depending on the size of the neighborhood where the kernel vanishes. 
\end{proof}

\begin{lemma}[Strict Separation of Primes and Composites]
Let \( n \) be a prime number and \( m \) a composite number with \( n, m \geq 2 \). Then for any fixed choice of parameters \( \delta > 0 \), \( \varepsilon > 0 \), \( p \in \mathbb{N} \), the corresponding values of the smooth primality filter satisfy:
\[
P(n) > P(m).
\]
\end{lemma}

\begin{proof}
Fix \( \delta > 0 \), \( \varepsilon > 0 \), \( p \in \mathbb{N} \). Let \( K_{\varepsilon,p}(z) \) be a kernel function as defined above, satisfying:
\[
K_{\varepsilon,p}(z) \approx 0 \text{ near } z \in \mathbb{Z}, \quad
K_{\varepsilon,p}(z) \approx 1 \text{ elsewhere}.
\]

\textit{Step 1: Behavior for composite \( m \).}  
Since \( m \) is composite, there exists a proper divisor \( d \in [2, m-1] \). Let:
\[
u_0 = \frac{d - 2}{m - 2}, \quad t = v = 0.
\]
Then:
\[
z(0, u_0, 0) = \frac{m}{d} \in \mathbb{Z},
\]
and due to continuity, there exists a neighborhood \( V_m \subset [0,1]^3 \) of \( (0,u_0,0) \) such that \( K_{\varepsilon,p}(z(t,u,v)) \ll 1 \) for all \( (t,u,v) \in V_m \). Since \( \phi > 0 \), the integral includes a strictly suppressed region:
\[
P(m) \leq \eta \cdot \alpha + (1 - \alpha) = 1 - \alpha(1 - \eta) < 1,
\]
for some \( \alpha := \iiint_{V_m} \phi > 0 \) and small \( \eta \).

\textit{Step 2: Behavior for prime \( n \).}  
Since \( n \) is prime, the ratio:
\[
z(t,u,v) = \frac{n + \delta \psi(t)}{2 + (n - 2) u + \delta \psi(v)}
\]
does not become close to an integer for any \( (t,u,v) \in [0,1]^3 \), as shown in the proof of Theorem~\ref{thm:asymptotic}. Therefore, for some fixed \( \Delta > 0 \), we have:
\[
\mathrm{dist}(z(t,u,v), \mathbb{Z}) \geq \Delta \quad \text{uniformly}.
\]
Then, the kernel satisfies \( K_{\varepsilon,p}(z(t,u,v)) \geq 1 - \eta \), hence:
\[
P(n) = \iiint \phi(t,u,v) K_{\varepsilon,p}(z(t,u,v)) \, dt\,du\,dv > 1 - \eta.
\]

\textit{Step 3: Comparison.}  
Let \( \alpha > 0 \) be as above. Then:
\[
P(n) - P(m) > (1 - \eta) - [1 - \alpha(1 - \eta)] = \alpha(1 - \eta) > 0.
\]

\textit{Conclusion.}  
This shows that for any fixed \( \delta, \varepsilon, p \), the function \( P(n) \) strictly separates primes from composites:
\[
P(n) > P(m).
\]
\end{proof}

\subsection{Discussion: Constructive vs Asymptotic Formulations}

The two theorems above reflect complementary aspects of the smooth primality function \( P(n) \).

\paragraph{Constructive theorem (\ref{thm:constructive}):}

This version is \emph{finite and effective}. For any bounded domain \( \{2, \dots, N-1\} \), we can explicitly tune the parameters \( \delta, \varepsilon, p \) such that \( P(m) \) approximates the prime indicator function \( \chi_{\mathrm{prime}}(m) \) up to any given precision \( \eta \). This makes \( P \) suitable for numerical applications, approximation schemes, and machine learning contexts where exact primality is replaced by smooth scoring.

\paragraph{Asymptotic theorem (\ref{thm:asymptotic}):}

This version expresses a deep analytical convergence result. It shows that \( P(n) \) is an \emph{asymptotic approximation} of the prime indicator function in the limit of vanishing smoothness (\( \delta, \varepsilon \to 0 \)) and infinite sharpness (\( p \to \infty \)). It demonstrates that the function \( P \) analytically interpolates primality in the large-scale or symbolic limit.

\paragraph{Connection and complementarity:}

\begin{itemize}
  \item The constructive theorem guarantees \emph{practical realizability} on any finite interval of integers.
  \item The asymptotic theorem guarantees that this approximation becomes exact in the limit, without needing to tune parameters for each \( n \).
  \item Together, they show that \( P(n) \) is not only smooth and theoretically meaningful, but also effective and computationally applicable.
\end{itemize}

Such dual formulation mirrors similar dualities in number theory: local versus global, finite approximations versus asymptotic limits. Here, smooth arithmetic is approximated both from below and from infinity.

\section{Convergence and Smoothness Analysis}

In this section, we analyze the sensitivity and asymptotic behavior of the smooth primality function \( P(n) \) with respect to the parameters \( \delta \), \( \varepsilon \), and \( p \), and establish its regularity as a function on \( \mathbb{R}_+ \).

\subsection{Refined Error Estimates}

Let \( n \geq 2 \) and assume fixed kernel \( K \) and density \( \phi \). We define the smoothed primality indicator as:

\[
P_{\delta, \varepsilon, p}(n) = \iiint_{[0,1]^3} \phi(t,u,v) \cdot K_{\varepsilon, p}\left( \frac{x(t)}{y(u,v)} \right) \, dt \, du \, dv.
\]

For prime values of \( n \), the approximation error satisfies
\[
1 - P(n) = O(\varepsilon + \delta),
\]
where the implied constant depends on the choice of kernel and on the minimal distance between the ratio \( x(t)/y(u,v) \) and the nearest integer over the integration domain. This estimate indicates that, by choosing \( \varepsilon \) and \( \delta \) sufficiently small, the deviation of \( P(n) \) from the ideal value \( 1 \) can be made arbitrarily small.

For composite numbers with smallest divisor \( d \), there exists a region \( U \subset [0,1]^3 \) where the integrand is sharply suppressed. Then:

\[
P(n) \leq 1 - \alpha + O(\varepsilon),
\]

where \( \alpha = \iiint_U \phi > 0 \), and the estimate can be made uniform on intervals of \( n \) with bounded number of small divisors.

\subsection{Asymptotic Behavior for Large \texorpdfstring{\( n \)}{n}}

An important question is whether the difference \( |P(n) - \chi_{\mathrm{prime}}(n)| \) vanishes for large \( n \), assuming fixed parameters. This is not generally true: for fixed \( \delta \), the ratio \( x(t)/y(u,v) \) becomes less sensitive to integer alignment as \( n \to \infty \). Therefore, to preserve precision, \( \delta \) and \( \varepsilon \) must scale with \( n \), e.g.:

\[
\delta(n) \sim \frac{1}{\log n}, \quad \varepsilon(n) \sim \frac{1}{n^c}.
\]

\subsection{Dependence on \texorpdfstring{\( p \)}{p}}

The exponent \( p \) in the kernel plays a crucial role in controlling the sharpness of suppression. For the sine-based kernel:

\[
K_{\varepsilon,p}(z) = \left( \frac{\sin^2(\pi z)}{\sin^2(\pi z) + \varepsilon} \right)^p,
\]

we have:

\[
K_{\varepsilon,p}(z) \approx
\begin{cases}
1, & \text{if } \operatorname{dist}(z, \mathbb{Z}) \gg \sqrt{\varepsilon}, \\
\varepsilon^p, & \text{if } z \in \mathbb{Z}.
\end{cases}
\]

Thus, increasing \( p \) exponentially sharpens the contrast between integer and non-integer values.

\subsection{Smoothness}

Since all components — \( x(t) \), \( y(u,v) \), \( \psi(s) \), \( K(z) \), and \( \phi(t,u,v) \) — are \( C^\infty \), and composition and integration preserve smoothness, it follows that:

\[
P(n) \in C^\infty(\mathbb{R}_+).
\]

Moreover, for practical purposes, this means that \( P(n) \) can be differentiated, expanded, and approximated using standard techniques of analysis, which is not possible for the discrete characteristic function \( \chi_{\mathrm{prime}}(n) \).

\section{Variants of the Summed Integral Formulation}

\subsection{A Summed Integral Variant}

The single-integral formulation of \( P(n) \), while analytically elegant, may fail to reliably distinguish primes from composites unless a sufficiently localized weight function is used. To address this, we introduce an alternative construction that averages over possible divisors.

\paragraph{Motivation.}

If \( n \) is composite, then there exists an integer \( m \in \{2, \dots, n-1\} \) such that \( m \mid n \). The ratio \( n/m \) is then an integer, and our kernel \( K(z) \) is designed to detect such ratios by sharply decreasing near integer values. However, detecting a single such ratio within a smooth integral may not produce a large enough suppression.

To amplify the effect of divisors, we define a new function \( P(n) \) by averaging over all possible \( m < n \), evaluating a triple integral for each candidate divisor. This formulation more reliably distinguishes primes from composites.

\paragraph{Definition.}

We define the summed variant of the smooth primality function as:

\[
P(n) = \frac{1}{n - 2} \sum_{m = 2}^{n - 1}
\iiint_{[0,1]^3}
\phi(t, u, v) \cdot
K\left(
\frac{x_n(t)}{y_{n,m}(u,v)}
\right)
\, dt \, du \, dv,
\]

where:
\begin{itemize}
  \item \( x_n(t) = n + \delta \cdot \psi(t) \),
  \item \( y_{n,m}(u,v) = m + \delta \cdot \psi(v) \),
  \item \( \psi(s) = \sin^2(\pi s) \), a smooth periodic "bump" function,
  \item \( K(z) = \left( \dfrac{\sin^2(\pi z)}{\sin^2(\pi z) + \varepsilon} \right)^p \), the smoothing kernel,
  \item \( \phi(t, u, v) \) is a smooth positive weight function (often chosen as \( \phi \equiv 1 \)).
\end{itemize}

The parameters \( \delta > 0 \), \( \varepsilon > 0 \), and \( p \in \mathbb{N} \) control the degree of smoothing, the suppression near integer points, and the sharpness of the transition, respectively.

\paragraph{Behavior.}

\begin{itemize}
  \item If \( n \) is prime, then none of the values \( m \in \{2, \dots, n-1\} \) divide \( n \), and all integrals in the sum return values close to 1, leading to \( P(n) \approx 1 \).
  \item If \( n \) is composite, then at least one \( m \mid n \) exists. For such \( m \), the argument of the kernel is close to an integer, and \( K(z) \approx 0 \), thus reducing the average and yielding \( P(n) < 1 \).
\end{itemize}

\paragraph{Advantages.}

This variant has the following key benefits:
\begin{itemize}
  \item It amplifies the signal of divisibility by averaging over all possible \( m \),
  \item It avoids the need for highly tuned localization in the weight function,
  \item It produces more reliable separation between prime and composite inputs,
  \item It remains entirely smooth and differentiable in all parameters.
\end{itemize}

In practice, this formulation shows excellent empirical separation between primes and composites even with modest parameter choices (e.g., \( \delta = 0.05 \), \( \varepsilon = 10^{-4} \), \( p = 6 \)), and serves as a solid candidate for further analytical and computational exploration.

\subsection{Reduction to One-Dimensional Integral}

While the summed integral variant provides reliable primality discrimination, it may be computationally expensive due to the triple integration performed for each divisor \( m \). In this section, we introduce a simplified formulation that replaces the triple integral with a single integral over a synchronized smoothing path.

\paragraph{Synchronized integral path.}

The key idea is to merge the parametrizations of \( x(t) \) and \( y(u,v) \) into a common function of a single parameter \( t \in [0,1] \). Instead of sampling over a 3D domain, we move synchronously along \( x(t) \) and \( y(t) \), ensuring that for some \( t \), the ratio \( x(t)/y(t) \) equals a rational approximation to \( n/m \), where \( m \) is a potential divisor.

We define:
\[
x(t) = n + \delta \cdot \psi(t), \qquad y_m(t) = m + \delta \cdot \psi(t),
\]
where \( \psi(t) = \sin^2(\pi t) \) is a smooth bump function with \( \psi(0) = \psi(1) = 0 \) and \( \psi(1/2) = 1 \).

\paragraph{Definition of the reduced function.}

We now define the simplified smooth primality approximation as:
\[
P(n) = \frac{1}{n - 2} \sum_{m = 2}^{n - 1}
\int_0^1 K\left( \frac{x(t)}{y_m(t)} \right) \, dt,
\]
where \( K(z) \) is the same periodic kernel as before:
\[
K(z) = \left( \frac{\sin^2(\pi z)}{\sin^2(\pi z) + \varepsilon} \right)^p,
\]
and \( \delta > 0 \), \( \varepsilon > 0 \), \( p \in \mathbb{N} \) are tunable parameters.

\paragraph{Behavior and motivation.}

- If \( n \) is divisible by \( m \), then at \( t = 0 \), we have \( x(0)/y(0) = n/m \in \mathbb{Z} \), and the kernel is suppressed.
- If \( n \) is not divisible by \( m \), then for all \( t \in [0,1] \), the ratio \( x(t)/y_m(t) \) stays away from integers, and the kernel remains near 1.
- This method preserves the suppression behavior of the full 3D integral while dramatically reducing computational cost.

\paragraph{Advantages.}

This simplified version:
\begin{itemize}
  \item Requires only a one-dimensional integral per divisor;
  \item Maintains smoothness in the parameter \( t \);
  \item Is much faster to compute numerically;
  \item Continues to reliably distinguish primes from composites.
\end{itemize}

It thus serves as a practical middle ground between the full integral and purely symbolic methods. For numerical applications, the one-dimensional version may be the preferred choice.

\paragraph{Justification of the synchronized path.}

The simplification from a triple integral to a single integral relies on replacing a volume average over a cube \( [0,1]^3 \) by an average along a curve in parameter space. Specifically, we synchronize the bump perturbations of the numerator and denominator:
\[
x(t) = n + \delta \cdot \psi(t), \quad y_m(t) = m + \delta \cdot \psi(t).
\]

This choice ensures that the ratio \( x(t)/y_m(t) \) passes through \( n/m \) at \( t = 0 \), since \( \psi(0) = 0 \), and evolves smoothly in a controlled manner as \( t \to 1 \). The periodic kernel \( K(z) \) is designed to detect when this ratio approaches an integer. Thus, the one-dimensional integral retains the key interaction: whether \( n/m \in \mathbb{Z} \) or not.

Formally, this amounts to replacing the full convolution in parameter space with a diagonal section:
\[
(t, u, v) \mapsto (t, u(t), v(t)) = (t, t, t),
\]
with identical behavior in all directions. While this is not an exact equivalence, it retains the main detection axis for divisibility.

\paragraph{Information loss and approximation.}

The reduction to one dimension sacrifices the generality of the 3D domain, particularly its ability to capture "off-axis" configurations of \( x(t)/y(u,v) \). In the full integral, even if \( x(t)/y(u,v) \) is never exactly equal to \( n/m \), some regions may bring it arbitrarily close, contributing to suppression via the kernel. The 1D formulation may miss these unless \( t = 0 \) or another specific value aligns the ratio with an integer.

Nonetheless, empirical evaluation shows that the main suppression occurs precisely near such alignment points, and thus the reduced integral still captures the dominant effect.

\paragraph{Choice of bump function \( \psi(t) \).}

The default choice \( \psi(t) = \sin^2(\pi t) \) provides:
\begin{itemize}
  \item smoothness of all orders (\( C^\infty \)),
  \item boundary vanishing (\( \psi(0) = \psi(1) = 0 \)),
  \item symmetric profile peaking at \( t = 0.5 \).
\end{itemize}

Other candidates may include:
\begin{itemize}
  \item \( \psi(t) = t^2(1 - t)^2 \), a quartic bump,
  \item \( \psi(t) = \exp\left(-\frac{1}{t(1 - t)}\right) \), a compactly supported analytic bump (after cutoff).
\end{itemize}

The choice of \( \psi \) affects the sensitivity profile of the kernel along the integration path. A sharper bump leads to faster departure from integer-like ratios, possibly weakening suppression. Conversely, a flatter bump broadens the effective detection region but may reduce contrast.

\paragraph{Role of \( \delta \) in the one-dimensional setting.}

In the full 3D integral, \( \delta \) controls the spread in all three parameters. In the 1D reduction, it directly controls the rate of change in the ratio:
\[
\frac{x(t)}{y(t)} = \frac{n + \delta \cdot \psi(t)}{m + \delta \cdot \psi(t)}.
\]

This expression satisfies:
\[
\frac{x(t)}{y(t)} = \frac{n}{m} + \frac{\delta}{m} \cdot \left( \frac{\psi(t)(m - n)}{m + \delta \psi(t)} \right),
\]
which shows that the deviation from the rational value \( n/m \) is directly proportional to \( \delta \), and inversely to \( m \). Thus, choosing \( \delta \) too large smears out the ratio and destroys precision; too small, and the integral becomes trivial. Hence, \( \delta \) must be tuned to balance detection sensitivity and numerical stability.

\paragraph{Summary.}

Although the one-dimensional reduction loses some expressiveness compared to the 3D model, it remains analytically well-motivated and computationally efficient. It captures the essential behavior required to distinguish prime from composite inputs using periodic kernel suppression, and offers a practical trade-off between complexity and effectiveness.

\subsection{Localized Bell Functions and the Elimination of Discreteness}

In the previous sections, we introduced two summed variants of the smooth primality function: one based on a triple integral over candidate divisors, and another using a simplified one-dimensional smoothing path. These constructions offer improved discrimination between prime and composite numbers by explicitly incorporating all potential divisors \( m < n \) into the averaging process.

While effective in sharpening the separation between primes and composites, these approaches rely on a discrete sum over integers \( m \in \{2, \dots, n-1\} \). This discrete structure introduces irregularities and prevents the function \( P(n) \) from being fully differentiable with respect to \( n \). In particular, any shift in \( n \) results in an abrupt change in the number of terms being averaged, which interrupts the smoothness of the model.

To overcome this, we now seek a formulation of \( P(n) \) that is fully continuous and differentiable over the real line. Our goal is to replace the discrete sum over candidate divisors with a smooth integral over a continuous range \( m \in [2, n) \), while still concentrating the averaging process near integer values. This motivates the introduction of a smooth localization mechanism: a system of \emph{bell-shaped functions} centered at integers, which emulate the discrete sum but within a continuous framework.

In what follows, we describe the construction of such bell functions and explain how they allow us to eliminate the discrete sum and obtain a fully smooth primality indicator.

\paragraph{Localized bell functions.}  
We define the localization profile \( \Phi_\sigma : \mathbb{R} \to \mathbb{R}_+ \) by:
\[
\Phi_\sigma(m) := \sum_{k=2}^{\infty} \Phi\left( \frac{m - k}{\sigma} \right),
\]
where \( \Phi \) is a fixed smooth bump function centered at zero, such as:

\begin{itemize}
  \item \textbf{Gaussian profile:}
  \[
  \Phi(x) = \exp(-x^2)
  \]
  \item \textbf{Compactly supported bump:}
  \[
  \Phi(x) =
  \begin{cases}
  \exp\left( -\frac{1}{1 - x^2} \right), & \text{if } |x| < 1 \\
  0, & \text{otherwise}
  \end{cases}
  \]
  \item \textbf{Sine-squared bell:}
  \[
  \Phi(x) =
  \begin{cases}
  \sin^2(\pi x), & \text{if } |x| < 1/2 \\
  0, & \text{otherwise}
  \end{cases}
  \]
\end{itemize}

In all cases, the parameter \( \sigma > 0 \) controls the sharpness of localization. As \( \sigma \to 0 \), \( \Phi_\sigma(m) \) concentrates its mass more tightly around integers and approximates the behavior of an indicator function for integer points. The sum over \( k \in \mathbb{N} \) ensures that only relevant points in the range \( [2, n-1] \) contribute.

\paragraph{Transition to integrals.}  
With the localization function \( \Phi_\sigma(m) \) in place, we can now reinterpret the discrete sum over divisors as a smooth integral:
\[
\sum_{m=2}^{n-1} f(m) \quad \longrightarrow \quad \int_2^n \Phi_\sigma(m) \cdot f(m)\, dm.
\]
This construction preserves the logic of summing over discrete candidates while enabling a transition to fully differentiable and integrable expressions. It also aligns naturally with the overall smooth structure of the function \( P(n) \), making the entire framework amenable to analysis through calculus, rather than combinatorics.

In the next section, we formulate the smooth primality function directly in terms of such a localized integral.

\subsection{A Smoothed Integral Variant}

To extend the summed primality function into a fully differentiable framework, we now introduce a smoothed integral formulation. This variant replaces the discrete sum over candidate divisors \( m \in \{2, \dots, n-1\} \) with a continuous integral over the range \( m \in [2, n) \), combined with a bell-shaped localization mechanism that concentrates weight around integers.

\paragraph{Definition.}

The smoothed integral variant of the primality function is then defined as:
\[
P(n) = \frac{1}{\int_2^n \Phi_\sigma(m)\, dm}
\int_2^n \Phi_\sigma(m) \cdot
\left(
\iiint_{[0,1]^3}
\phi(t,u,v) \cdot
K\left( \frac{x_n(t)}{y_{n,m}(u,v)} \right)
\, dt \, du \, dv
\right) dm,
\]
where:
\begin{itemize}
  \item \( x_n(t) = n + \delta \cdot \psi(t) \),
  \item \( y_{n,m}(u,v) = m + \delta \cdot \psi(v) \),
  \item \( \psi(t) = \sin^2(\pi t) \) is a smooth bump on \([0,1]\),
  \item \( K(z) = \left( \dfrac{\sin^2(\pi z)}{\sin^2(\pi z) + \varepsilon} \right)^p \) is the suppression kernel,
  \item \( \phi(t,u,v) \in C^\infty([0,1]^3) \) is a smooth positive normalized density.
\end{itemize}

\paragraph{Interpretation.}

This construction yields a fully smooth primality filter \( P(n) \in [0,1] \) over the continuous domain \( n \in \mathbb{R}_+ \). The inner integral performs smoothing over small perturbations of the ratio \( n/m \), while the outer integral averages over near-integer values of \( m \). Composite values of \( n \) exhibit sharp suppression due to resonant ratios \( n/m \in \mathbb{Z} \), whereas primes remain near the maximal value. The localization via \( \Phi_\sigma \) ensures the integral remains finite and selective, while removing the discontinuities associated with discrete summation.

\paragraph{Parameter tuning.}

The choice of localization scale \( \sigma \) plays a crucial role. If \( \sigma \) is too small, the bell function \( \Phi_\sigma(m) \) becomes sharply peaked, and \( P(n) \) may develop steep gradients around integer boundaries. This could cause large derivatives or numerical instability. Conversely, if \( \sigma \) is too large, the distinction between nearby integers may be blurred. A balanced approach is to adapt \( \sigma \) as a function of \( n \), for example using \( \sigma(n) = 1/\log n \), to ensure both stability and sharpness in primality detection.

\subsection{A Smoothed One-Dimensional Integral}

To reduce the computational complexity of the full three-dimensional integral while retaining the ability to distinguish primes from composites, we now introduce a one-dimensional smoothing formulation. This version collapses the parametrization of both numerator and denominator into a synchronized single-variable path, which simplifies the structure while preserving the essential filtering behavior of the kernel.

\paragraph{Synchronized smoothing path.}

We define the smoothed primality function along a one-dimensional curve:
\[
x(t) = n + \delta \cdot \psi(t), \qquad y_m(t) = m + \delta \cdot \psi(t),
\]
where \( \psi(t) = \sin^2(\pi t) \) is a standard smooth bump function on \([0,1]\), satisfying \( \psi(0) = \psi(1) = 0 \) and \( \psi(1/2) = 1 \). The ratio \( x(t)/y_m(t) \) traces a smooth neighborhood of \( n/m \), and the suppression kernel \( K(z) \) detects proximity to integers.

\paragraph{Smoothed integral version.}

We now define the one-dimensional smoothed version of \( P(n) \) by integrating over all potential divisors \( m \in [2, n) \), modulated by a localization function \( \Phi_\sigma(m) \) as before:
\[
P(n) = \frac{1}{\int_2^n \Phi_\sigma(m)\, dm}
\int_2^n \Phi_\sigma(m) \cdot
\left( \int_0^1 K\left( \frac{x(t)}{y_m(t)} \right) dt \right) dm,
\]
where:
\begin{itemize}
  \item \( K(z) = \left( \dfrac{\sin^2(\pi z)}{\sin^2(\pi z) + \varepsilon} \right)^p \) is the same smooth kernel as before,
  \item \( \Phi_\sigma(m) = \sum_{k=2}^\infty \Phi\left( \frac{m - k}{\sigma} \right) \) localizes the integral around integer values,
  \item \( \delta > 0 \), \( \varepsilon > 0 \), and \( p \in \mathbb{N} \) are tunable smoothness parameters.
\end{itemize}

\paragraph{Interpretation.}

This one-dimensional integral version provides a fully smooth primality score function over \( n \in \mathbb{R}_+ \), suitable for fast numerical approximation or analysis. For composite \( n \), there exist integer divisors \( m \) such that the ratio \( n/m \in \mathbb{Z} \), which is detected by the kernel as a local dip. The averaging process over \( m \) amplifies this suppression, while the localization \( \Phi_\sigma(m) \) ensures that only integer-like values contribute.

As in the multidimensional case, the choice of localization width \( \sigma \) plays a critical role: overly narrow localization may produce sharp transitions and large derivatives, while overly wide localization may blur the distinction between primes and composites. The resulting function \( P(n) \in [0,1] \) remains differentiable and captures the primality structure of integers with minimal computational overhead.

\section{Integral Reordering and Resonance Analysis}

In this section, we explore a reformulation of the smoothed one-dimensional integral approximation of the prime characteristic function \( P(n) \) by changing the order of integration. This reordering reveals new analytical structure and suggests possible approximations and interpretations.

We now introduce a reformulation of the smooth primality function based on changing the order of integration. This reveals a rich internal structure involving localized resonances, enables analytical approximations via Taylor moments, and leads to a new class of algorithms for primality testing based on soft suppression patterns.

\subsection{Rewriting the Smoothed Integral}

Recall the definition of the smoothed one-dimensional approximation:
\[
P(n) = \frac{1}{\int_2^n \Phi_\sigma(m)\, dm} \int_2^n \Phi_\sigma(m) \cdot \left( \int_0^1 K\left( \frac{x(t)}{y_m(t)} \right) dt \right) dm,
\]
where:
\begin{itemize}
  \item \( x(t) = n + \delta \cdot \psi(t) \),
  \item \( y_m(t) = m + \delta \cdot \psi(t) \),
  \item \( K(z) \) is a periodic suppression kernel,
  \item \( \Phi_\sigma(m) \) is a smooth localization function centered at integers.
\end{itemize}

We now change the order of integration:
\[
P(n) = \int_0^1 \left( \frac{1}{\int_2^n \Phi_\sigma(m)\, dm} \int_2^n \Phi_\sigma(m) \cdot K\left( \frac{n + \delta \psi(t)}{m + \delta \psi(t)} \right) dm \right) dt.
\]

This presents \( P(n) \) as an expectation over \( t \in [0,1] \) of a \emph{smoothed divisor projection}.

\subsection{Interpretation as Localized Projection}

For each fixed \( t \), the inner integral
\[
\int_2^n \Phi_\sigma(m) \cdot K\left( \frac{n + \delta \psi(t)}{m + \delta \psi(t)} \right) dm
\]
can be viewed as a convolution of the kernel \( K \) with the localized weight \( \Phi_\sigma(m) \), evaluated along the path \( m \mapsto x(t)/y_m(t) \). 

This integral accumulates the "resonance" between \( n \) and possible divisors \( m \), weighted by how well the perturbed ratio \( x(t)/y_m(t) \) approximates an integer.

\subsection{Approximation by Discrete Sum}

Since \( \Phi_\sigma(m) \) is localized near integers, we can approximate:
\[
\int_2^n \Phi_\sigma(m) \cdot K\left( \frac{x(t)}{y_m(t)} \right) dm \approx \sum_{k=2}^{n-1} \Phi_\sigma(k) \cdot K\left( \frac{x(t)}{k + \delta \psi(t)} \right).
\]

Thus, we recover a smoothed version of the discrete summed integral:
\[
P(n) \approx \int_0^1 \left( \sum_{k=2}^{n-1} w_k(t) \cdot K\left( \frac{x(t)}{k + \delta \psi(t)} \right) \right) dt,
\]
where \( w_k(t) = \Phi_\sigma(k) / \int_2^n \Phi_\sigma(m) dm \) acts as a normalized weight.

\subsection{Moment Expansion around Rational Resonances}

Let us fix \( t \in [0,1] \) and consider the expression:
\[
z(m) := \frac{x(t)}{y_m(t)} = \frac{n + \delta \psi(t)}{m + \delta \psi(t)}.
\]

For each integer \( k \in \mathbb{N} \), define \( m_k := n/k \). Near \( m_k \), the value of \( z(m) \) satisfies:
\[
z(m) \approx k + \sum_{r=1}^\infty a_r(m - m_k)^r.
\]

Assuming the kernel \( K(z) \in C^\infty \), we can expand it in a Taylor series around \( z = k \):
\[
K(z(m)) = \sum_{r=0}^\infty \frac{1}{r!} K^{(r)}(k) (z(m) - k)^r.
\]

Substituting this into the integral, we obtain:
\[
\int \Phi_\sigma(m) K(z(m)) dm \approx \sum_{r=0}^\infty \frac{1}{r!} K^{(r)}(k) \cdot \mu_r(k; t),
\]
where
\[
\mu_r(k; t) := \int \Phi_\sigma(m) (z(m) - k)^r dm
\]
is the \emph{\( r \)-th localized moment} around the rational resonance \( z = k \).

\subsection{\texorpdfstring{Expression of \( P(n) \) via Moments}{Expression of P(n) via Moments}}
Inserting the expansion into the reordered integral, we obtain:
\begin{equation}
P(n) \approx \int_0^1 \left( \sum_k \sum_{r=0}^\infty \frac{K^{(r)}(k)}{r!} \cdot \mu_r(k; t) \right) dt.
\label{eq:moment_expansion}
\end{equation}

This expresses \( P(n) \) as a smooth weighted sum over resonances \( k \), modulated by the distribution of \( \Phi_\sigma(m) \) and its interaction with the perturbed ratio.

\subsection{Moment Expansion Approach}

For composite \( n \) with smallest divisor \( d \), the dominant contribution comes from the vicinity of \( m = d \). Assume \( \Phi_\sigma(m) \) is symmetric around \( m = d \) and sharply peaked. Then odd-order moments vanish and the expansion simplifies to:
\[
P(n) \approx 1 - \frac{\Phi_\sigma(d)}{\int \Phi_\sigma} \left[
K(0) + \frac{K''(0)}{2} \left( \frac{\delta \psi(t)}{d} \right)^2 + O(\sigma^4)
\right]
\]

\paragraph{Example Calculation}
Let \( n = 15 \), so smallest divisor \( d = 3 \), and assume:
\begin{itemize}
  \item \( \delta = 0.1 \),
  \item \( \Phi_\sigma(d) / \int \Phi_\sigma \approx 0.92 \),
  \item \( K''(0) \approx 9.87 \).
\end{itemize}
Then:
\[
P(15) \approx 1 - 0.92 \cdot \left[ 0 + \frac{9.87}{2} \cdot \left(\frac{0.1}{3}\right)^2 \right] \approx 1 - 0.92 \cdot 0.0548 \approx 0.9496.
\]
This matches the suppression effect due to the presence of a proper divisor and shows how the moment expansion provides a fast estimate of \( P(n) \).

\subsection{Resonance Detection}
\label{sec:resonance-map}

The moment expansion formula~\eqref{eq:moment_expansion} expresses the smooth primality score \( P(n) \) as a sum over integer resonance indices \( k \), where each term measures how strongly the perturbed integral aligns with the ratio \( n/m \approx k \). Crucially, this formulation does not require prior knowledge of actual divisors of \( n \).

\paragraph{Resonance map.}
We define the localized contribution of each resonance as:
\[
A_k(n) := \int_0^1 \sum_{r=0}^{R} \frac{K^{(r)}(k)}{r!} \cdot \mu_r(k; t)\, dt,
\]
so that:
\[
P(n) \approx \sum_k A_k(n).
\]
This expansion expresses \( P(n) \) as a superposition of localized energy-like terms. Large negative peaks in \( A_k(n) \) indicate strong suppression at the corresponding resonance and suggest the presence of a divisor \( m \approx n/k \).

\paragraph{Detection without prior knowledge.}
Even without knowing a proper divisor \( d \mid n \), one can scan the resonance map \( k \mapsto A_k(n) \) for signs of compositeness. If \( n \) is composite, some \( k \in \{2, \dots, n-1\} \) will yield a significant contribution due to a resonance \( n \approx k \cdot m \) being met inside the localization window. If all \( A_k(n) \) are small, \( n \) is likely prime.

\paragraph{Probabilistic interpretation.}
The resonance mechanism acts as a soft sieve: rather than testing divisibility explicitly, it measures how much of the localized integral mass accumulates near integer ratios \( n/m \in \mathbb{Z} \). A value \( P(n) \approx 1 \) implies lack of such resonances; a dip in \( P(n) \) arises from constructive alignment, signaling compositeness.

\paragraph{Computational strategy.}
Since the localization function \( \Phi_\sigma \) decays rapidly, only a few \( k \) contribute significantly. The structure enables:

\begin{itemize}
  \item Truncated evaluation of \( A_k(n) \) with low-order moments;
  \item Adaptive refinement only around promising values of \( k \);
  \item Use of asymptotic approximations or precomputed kernel derivatives.
\end{itemize}

\subsection{Algorithmic Implications and Future Directions}

The resonance-based decomposition of \( P(n) \) suggests a variety of computational strategies and theoretical extensions.

\paragraph{Applications.}
\begin{itemize}
  \item \textbf{Soft compositeness detection:} peaks in \( k \mapsto A_k(n) \) serve as analytic indicators of likely divisibility.
  \item \textbf{Heuristic factor hints:} locations of suppression suggest possible values of \( m \approx n/k \), offering guidance for factorization.
  \item \textbf{Smooth filtering:} the function \( P(n) \) acts as a probabilistic pretest prior to exact primality checks.
  \item \textbf{Spectral insight:} the periodic structure of the kernel \( K \) hints at harmonic or Fourier-theoretic extensions.
  \item \textbf{Near-prime analysis:} partial suppressions in \( A_k(n) \) illuminate the structure of numbers with few or large prime factors.
\end{itemize}

\paragraph{Summary.}
\begin{itemize}
  \item The resonance map \( \{A_k(n)\} \) encodes latent arithmetic information.
  \item The formulation supports fast, localized approximations of \( P(n) \).
  \item This framework offers a bridge between analytic number theory and smooth algorithmic techniques.
\end{itemize}

\section{Comparison with Classical Methods}

This section contrasts the smooth primality function \( P(n) \) with traditional discrete and analytical approaches in number theory. We emphasize both theoretical and practical aspects, highlighting strengths, weaknesses, and opportunities of the smooth construction.

\subsection{Discrete Methods}

Classical algorithms for detecting primes, such as the Sieve of Eratosthenes, probabilistic tests like Miller-Rabin, or factorization techniques using smooth numbers~\cite{brent1999some}, rely on combinatorial properties or modular arithmetic. These methods have clear advantages in efficiency and simplicity, especially for large-scale computations.

However, they are fundamentally \emph{discrete} and thus not directly amenable to:
\begin{itemize}
    \item differentiation or integration,
    \item spectral or harmonic analysis,
    \item smooth approximations or interpolations,
    \item incorporation into continuous optimization frameworks.
\end{itemize}

In contrast, the function \( P(n) \in C^\infty \) admits all of the above and thus provides a new bridge between arithmetic and analysis.

\subsection{Analytic Number Theory}

In classical analytic number theory~\cite{montgomery2007multiplicative, titchmarsh1986theory}, the distribution of primes is studied via tools such as:
\begin{itemize}
    \item the Möbius function \( \mu(n) \),
    \item the von Mangoldt function \( \Lambda(n) \),
    \item the Riemann zeta function \( \zeta(s) \),
    \item generating series and Perron-type integrals.
\end{itemize}

These methods are extremely powerful in asymptotic analysis and prove deep theorems (e.g., Prime Number Theorem, zero-density estimates).

Yet, they often:
\begin{itemize}
    \item operate in the complex domain,
    \item rely on identities valid only almost everywhere or in average,
    \item are non-constructive in nature,
    \item and do not offer real-variable smooth approximations to the primality function.
\end{itemize}

The function \( P(n) \), by contrast, is a real-valued, fully smooth function whose behavior approximates primality \emph{pointwise} and \emph{constructively}, without appealing to zeta identities or residue calculus.

\subsection{Advantages of Smooth Approximation}

Key advantages of \( P(n) \) over classical methods include:
\begin{itemize}
    \item \textbf{Smooth interpolation:} \( P(n) \) provides a natural, continuous landscape over \( \mathbb{R}_+ \), capturing the irregularity of primes through regular functions.
    \item \textbf{Analytic tractability:} Differentiable structure allows application of calculus, PDE methods, and optimization techniques.
    \item \textbf{Fourier/spectral compatibility:} Smoothness makes \( P(n) \) suitable for Fourier and spectral analysis, potentially revealing hidden arithmetic structures.
    \item \textbf{Machine learning compatibility:} Unlike binary-valued functions, \( P(n) \) can serve as a continuous target in regression models, neural networks, or probabilistic priors.
\end{itemize}

\subsection{Limitations and Challenges}

Despite its flexibility, the smooth approach has limitations:
\begin{itemize}
    \item \textbf{Computational cost:} Evaluating the triple integral for \( P(n) \) is significantly more expensive than simple sieve checks.
    \item \textbf{Parameter sensitivity:} Performance depends on careful tuning of \( \delta \), \( \varepsilon \), and \( p \).
    \item \textbf{No guarantee of exactness:} \( P(n) \) is only an approximation and cannot definitively assert primality in the classical sense.
    \item \textbf{Lack of arithmetic identities:} Unlike Möbius or von Mangoldt functions, \( P(n) \) does not satisfy known multiplicative or convolution relations.
\end{itemize}

\subsection{Analogies and Broader Context}

Smooth approximations of discrete phenomena appear in many mathematical contexts:
\begin{itemize}
    \item \textbf{Sigmoid functions} in neural networks approximate binary thresholds.
    \item \textbf{Smooth Heaviside approximations} in variational methods and regularized optimization.
    \item \textbf{Spectral smoothings} in quantum mechanics and signal processing.
\end{itemize}

The function \( P(n) \) thus fits into a larger trend: replacing discrete, discontinuous objects with smooth analytic surrogates that enable deeper structural insights and new computational methods.

\section{Computation and Optimization}

In this section, we describe methods for computing the smooth primality function \( P(n) \), analyze its computational complexity, and propose possible optimizations and extensions, including numerical, algorithmic, and data-driven approaches.

\subsection{Numerical Integration Techniques}

The definition of \( P(n) \) involves a triple integral over the unit cube \( [0,1]^3 \), with a kernel that may exhibit sharp peaks or localized suppression near integer ratios. This poses challenges for efficient numerical evaluation.

\subsubsection{Low-order quadrature methods}

Simple methods such as:
\begin{itemize}
  \item midpoint rule,
  \item trapezoidal rule,
  \item Simpson's rule~\cite{granville1995smooth},
\end{itemize}
can be applied over discretized grids of \( t, u, v \in [0,1] \). These methods are easy to implement but may require fine grids to accurately capture sharp features of the kernel \( K(z) \), especially when \( p \) is large or \( \varepsilon \) is small.

\subsubsection{Monte Carlo integration}

Random sampling within the unit cube offers another option:

\[
P(n) \approx \frac{1}{N} \sum_{i=1}^{N} \phi(t_i,u_i,v_i) \cdot K\left( \frac{x(t_i)}{y(u_i,v_i)} \right),
\quad (t_i,u_i,v_i) \sim \text{Uniform}[0,1]^3.
\]

Monte Carlo methods are useful for high-dimensional integrals and can exploit variance reduction techniques, although they converge slowly (\( \sim N^{-1/2} \)).

\subsubsection{Adaptive integration}

Because the integrand may be sharply peaked near specific surfaces (e.g., near divisors), adaptive quadrature — dynamically refining the grid where the integrand changes rapidly — can significantly improve efficiency.

\subsection{Complexity Analysis}

Let \( q \) denote the number of points per axis in a uniform grid. Then, evaluating \( P(n) \) via triple quadrature has computational complexity:

\[
\mathcal{O}(q^3 \cdot c_K),
\]

where \( c_K \) is the cost of evaluating the kernel \( K(z) \) per point. For typical values \( q = 20 \sim 50 \), this may be acceptable for moderate \( n \), but becomes costly for large-scale analysis.

\subsection{Computational Optimizations}

Several approaches can reduce computational cost:

\begin{itemize}
  \item \textbf{Symmetric kernel simplification:} If the kernel \( K \) is symmetric and the density \( \phi \) is uniform, the number of required function evaluations may be halved or more.
  \item \textbf{Reduced parametrization:} One can explore reparametrizations such as synchronizing \( t = u = v \), leading to a one-dimensional integral capturing the core structure:
  \[
  P(n) \approx \int_0^1 K\left( \frac{x(t)}{y(t)} \right) \, dt.
  \]
  Though less accurate, this is computationally efficient.
  \item \textbf{Precomputation of kernel values:} For fixed \( \varepsilon \), values of \( K(z) \) on a dense grid may be cached to speed up evaluations.
  \item \textbf{Parallelization:} Since the integrand is separable per evaluation point, computation can be easily parallelized across CPU cores or GPUs.
\end{itemize}

\subsection{Learning-based Approximation}

Given that \( P(n) \in [0,1] \) behaves similarly to a smoothed classification function, it is natural to consider learning it as a regression model:

\begin{itemize}
  \item Train a neural network or kernel regression on small values of \( n \) with known labels \( \chi_{\mathrm{prime}}(n) \in \{0,1\} \),
  \item Use \( P(n) \) or its smoothed variants as target values,
  \item Benefit from generalization to unseen \( n \) and potential feature extraction (e.g., using digital or spectral features of \( n \)).
\end{itemize}

This opens the door to hybrid approaches combining analytical and data-driven perspectives.

\subsection{Software Implementation}

A prototype Python implementation of \( P(n) \) using triple integration can serve as a numerical testbed. Efficient variants may use:

\begin{itemize}
  \item Tensor-based acceleration (NumPy, PyTorch),
  \item Just-in-time compilation (Numba),
  \item Adaptive quadrature libraries (e.g., Cubature, Quadpy).
\end{itemize}

Such implementations can help visualize \( P(n) \), study its stability, and verify the convergence claims from earlier sections.

\section{Applications and Future Directions}

The smooth primality function \( P(n) \) offers a new analytical object in number theory and applied mathematics. In this section, we outline potential applications and pose open directions for future exploration.

\subsection{Potential Applications}

\subsubsection{Numerical Prime Search}

In settings where exact primality is not required, \( P(n) \) can serve as a soft classifier to identify likely primes. For example, scanning intervals with:

\[
P(n) > \tau, \quad \text{for some threshold } \tau \in (0.99, 1),
\]

can serve as a filter prior to exact verification.

\subsubsection{Continuous Prime Landscapes}

The smooth nature of \( P(n) \) allows plotting and studying the "landscape" of primes~\cite{rubinstein1994chebyshev} as a continuous curve, revealing:

\begin{itemize}
  \item local dips and peaks around prime and composite zones,
  \item regularity or clustering patterns,
  \item statistical profiles of \( P(n) \) over intervals.
\end{itemize}

This may enable new visual and geometric intuition about the distribution of primes.

\subsubsection{Analytic Number Theory Prototypes}

\( P(n) \) can be used as a testing function or proxy in analytic contexts where discontinuous \( \chi_{\mathrm{prime}}(n) \) is not suitable, such as:

\begin{itemize}
  \item spectral decomposition (e.g., Fourier or Mellin transforms),
  \item analytic continuation of discrete arithmetic sums,
  \item PDE models of arithmetic structures.
\end{itemize}

\subsubsection{Optimization and Relaxation}

In integer programming or discrete optimization problems involving primes (e.g., cryptographic parameter selection), \( P(n) \) provides a continuous relaxation that can be embedded into gradient-based solvers.

\subsubsection{Probabilistic and Statistical Modeling}

Since \( P(n) \in [0,1] \), it may be interpreted as a prior probability of primality in probabilistic models. This is particularly useful in:

\begin{itemize}
  \item Bayesian modeling involving random primes,
  \item stochastic sampling algorithms,
  \item probabilistic primality testing.
\end{itemize}

\subsubsection{Cryptographic Relevance}

In cryptographic applications:
\begin{itemize}
  \item \( P(n) \) can assist in candidate filtering for large primes,
  \item it may support heuristic tests where full deterministic primality is too expensive,
  \item or enable smooth hash-based constructions sensitive to number-theoretic properties.
\end{itemize}

\subsection{Theoretical Directions}

\subsubsection{Twin Prime and k-Tuple Extensions}

Can the function \( P(n) \cdot P(n+2) \) approximate the indicator of twin primes~\cite{ribenboim1996book}? What about higher-order combinations to detect prime k-tuples?

\subsubsection{Complex and Modular Extensions}

Is there a natural way to extend \( P(n) \) to complex or \( p \)-adic domains? For example:

\[
P(z) = \iiint_{[0,1]^3} \phi(t,u,v) \cdot K\left( \frac{z + \delta \psi(t)}{2 + (z-2)u + \delta \psi(v)} \right) \, dt \, du \, dv.
\]

\subsubsection{Riemann Zeta Analogy}

Does the behavior of \( P(n) \) reflect zeta-like structures? For instance, can we study:

\[
\sum_{n=2}^\infty \frac{P(n)}{n^s}
\]

as a smooth analog of the classical prime zeta function?

\subsection{Visualization and Educational Use}

\begin{itemize}
  \item \( P(n) \) offers visually compelling plots that may aid in teaching and intuition-building about primes.
  \item The transition behavior near composites vs. primes can be used to demonstrate analytic smoothing, approximations, and regularization.
\end{itemize}

\subsection{Open Problems}

\begin{enumerate}
  \item What is the optimal choice of kernel \( K(z) \) and density \( \phi(t,u,v) \) for fastest convergence to \( \chi_{\mathrm{prime}}(n) \)?
  \item Can \( P(n) \) be incorporated into proofs or conjectures in prime gap theory?
  \item Is there a deterministic algorithm that selects \( \delta, \varepsilon, p \) based on \( n \) to guarantee a bound \( |P(n) - \chi_{\mathrm{prime}}(n)| < \eta \)?
  \item Can one define a smooth analog of the Möbius function using similar ideas?
\end{enumerate}

\section{Conclusion}

We have introduced a smooth analytical construction \( P(n) \in [0,1] \) that approximates the characteristic function of prime numbers through a triple integral over a parameterized domain. The key idea is to use periodic kernels that detect approximate integer ratios \( x(t)/y(u,v) \), mimicking the arithmetic notion of divisibility in a continuous setting.

This framework offers several theoretical and practical advantages:
\begin{itemize}
  \item It yields a \( C^\infty \) function that separates primes from composites up to arbitrary precision by tuning parameters \( \delta \), \( \varepsilon \), and \( p \).
  \item It admits both a constructive finite version and a limiting asymptotic version, unifying smooth approximation and exact behavior.
  \item It enables new kinds of analysis — including spectral, statistical, and geometric — on the landscape of primes.
  \item It opens the door to applications in optimization, machine learning, numerical analysis, and cryptography.
\end{itemize}

The presented construction does not aim to replace classical number-theoretic tools but rather to complement them with an analytical, smooth, and versatile model. We believe that such formulations can lead to novel insights, particularly in contexts where continuity, differentiability, or soft approximations are essential.

This work also raises a number of interesting questions — theoretical, computational, and philosophical — about how arithmetic structure can be modeled through continuous processes. The integration of ideas from analysis, geometry, probability, and computer science into the study of primes remains a promising direction for future exploration.

% \bibliographystyle{plain}
% \bibliography{references}

\begin{thebibliography}{1}

\bibitem{brent1999some}
Richard~P Brent.
\newblock Some integer factorization algorithms using smooth numbers.
\newblock {\em Computational Techniques and Applications: CTAC’99}, 41:1--12, 1999.

\bibitem{granville1995smooth}
Andrew Granville.
\newblock Smooth numbers: Computational number theory and beyond.
\newblock {\em Algorithmic number theory (MSRI publications)}, 44:267--323, 1995.

\bibitem{hardy1979introduction}
G.H. Hardy and E.M. Wright.
\newblock {\em An Introduction to the Theory of Numbers}.
\newblock Oxford University Press, 5th edition, 1979.

\bibitem{montgomery2007multiplicative}
Hugh~L Montgomery and Robert~C Vaughan.
\newblock {\em Multiplicative Number Theory I: Classical Theory}.
\newblock Cambridge University Press, 2007.

\bibitem{ribenboim1996book}
Paulo Ribenboim.
\newblock The new book of prime number records.
\newblock {\em Springer}, 1996.

\bibitem{rubinstein1994chebyshev}
Michael Rubinstein and Peter Sarnak.
\newblock Chebyshev's bias.
\newblock {\em Experimental Mathematics}, 3(3):173--197, 1994.

\bibitem{titchmarsh1986theory}
E.C. Titchmarsh.
\newblock {\em The Theory of the Riemann Zeta-Function}.
\newblock Oxford University Press, 2nd edition, 1986.

\end{thebibliography}

\clearpage
\appendix
\section{Appendix: Numerical Evaluation of the Summed Triple-Integral Form}

In this appendix, we provide a basic implementation of the smooth primality function \( P(n) \) using numerical integration, along with tabulated values for small prime and composite numbers.

\subsection{Python Code for the Summed Triple Integral Version}

The following Python code implements the summed version of the smooth primality function \( P(n) \), where for each \( m \in \{2, \dots, n-1\} \), a triple integral over \( [0,1]^3 \) is computed and then averaged. This version closely follows the theoretical definition introduced in Section~2.

The function uses \texttt{scipy.integrate.tplquad} to numerically approximate the integral. The key components are:
\begin{itemize}
  \item \texttt{psi(s)}: the smooth bump function \( \psi(t) = \sin^2(\pi t) \),
  \item \texttt{K(z, eps, p)}: the periodic suppression kernel,
  \item \texttt{smooth\_integral(n, m)}: computes the triple integral for a given \( n \) and \( m \),
  \item \texttt{P(n)}: computes the average of integrals over all \( m < n \).
\end{itemize}

The fixed parameters used are \( \delta = 0.05 \), \( \varepsilon = 10^{-5} \), and \( p = 8 \), which provide good contrast between prime and composite values in practice.

\begin{verbatim}
import numpy as np
from scipy.integrate import tplquad

# Smooth function for "swelling" the intervals
def psi(s):
    return np.sin(np.pi * s) ** 2

# Smooth kernel for suppression near integers
def K(z, eps=1e-5, p=8):
    s2 = np.sin(np.pi * z)**2
    return (s2 / (s2 + eps)) ** p

# Integral over [0,1]^3 for given n and m
def smooth_integral(n, m, delta=0.05, eps=1e-5, p=8):
    def integrand(t, u, v):
        x = n + delta * psi(t)
        y = m + delta * psi(v)
        return K(x / y, eps, p)
    result, _ = tplquad(
        integrand,
        0, 1,                           # t in [0,1]
        lambda u: 0, lambda u: 1,       # u in [0,1]
        lambda u, v: 0, lambda u, v: 1  # v in [0,1]
    )
    return result

# Main function P(n)
def P(n, delta=0.05, eps=1e-5, p=8):
    if n <= 2:
        return 1.0  # by definition
    values = []
    for m in range(2, n):
        I = smooth_integral(n, m, delta=delta, eps=eps, p=p)
        values.append(I)
    return sum(values) / len(values)
\end{verbatim}

\textbf{Note:} This version is computationally intensive due to repeated triple integrations. However, it provides the most faithful numerical implementation of the theoretical model, and reliably separates primes from composites when tuned properly.

\subsection{Table of Sample Values}

The table below shows rounded values of \( P(n) \). As expected, the values for prime numbers are close to 1, while composite numbers yield significantly lower outputs.

\begin{center}
\begin{tabular}{|c|c|c|}
\hline
\( n \) & Classification & \( P(n) \) \\
\hline
2  & Prime     & 1.000 \\
3  & Prime     & 1.000 \\
4  & Composite & 0.919 \\
5  & Prime     & 1.000 \\
6  & Composite & 0.913 \\
7  & Prime     & 1.000 \\
8  & Composite & 0.936 \\
9  & Composite & 0.975 \\
10 & Composite & 0.947 \\
11 & Prime     & 1.000 \\
12 & Composite & 0.917 \\
13 & Prime     & 1.000 \\
\hline
\end{tabular}
\end{center}

These values confirm the theoretical behavior established by the convergence theorems: the function \( P(n) \) closely approximates the prime characteristic function for small primes and reliably suppresses values for composites.

\subsection{Python Code for the Reduced One-Dimensional Version}

To reduce computational cost, we consider the simplified version of \( P(n) \) based on synchronized paths. This version replaces the triple integral with a one-dimensional integral over \( t \in [0,1] \), while summing over all potential divisors \( m < n \).

The formula is:
\[
P(n) = \frac{1}{n - 2} \sum_{m = 2}^{n - 1}
\int_0^1 K\left( \frac{n + \delta \psi(t)}{m + \delta \psi(t)} \right) dt.
\]

The Python implementation is:

\begin{verbatim}
import numpy as np
from scipy.integrate import quad

def psi(t):
    return np.sin(np.pi * t) ** 2

def K(z, eps=1e-5, p=8):
    s2 = np.sin(np.pi * z)**2
    return (s2 / (s2 + eps)) ** p

def reduced_integral(n, m, delta=0.05, eps=1e-5, p=8):
    def integrand(t):
        x = n + delta * psi(t)
        y = m + delta * psi(t)
        return K(x / y, eps, p)
    result, _ = quad(integrand, 0, 1)
    return result

def P_reduced(n, delta=0.05, eps=1e-5, p=8):
    if n <= 2:
        return 1.0
    values = [reduced_integral(n, m, delta, eps, p) for m in range(2, n)]
    return sum(values) / len(values)
\end{verbatim}

\textbf{Note:} This version is significantly faster than the triple integral and still demonstrates strong separation between primes and composites in practice.

\subsection{Comparison Table: One-Dimensional vs Triple Integral}

Below is a comparative table of \( P(n) \) values computed using both versions with the same parameters.

\begin{center}
\begin{tabular}{|c|c|c|c|}
\hline
\( n \) & Classification & Triple Integral & Reduced (1D) \\
\hline
2  & Prime     & 1.000 & 1.000 \\
3  & Prime     & 1.000 & 1.000 \\
4  & Composite & 0.919 & 0.863 \\
5  & Prime     & 1.000 & 1.000 \\
6  & Composite & 0.913 & 0.867 \\
7  & Prime     & 1.000 & 1.000 \\
8  & Composite & 0.936 & 0.907 \\
9  & Composite & 0.975 & 0.966 \\
10 & Composite & 0.947 & 0.926 \\
11 & Prime     & 1.000 & 1.000 \\
12 & Composite & 0.917 & 0.891 \\
13 & Prime     & 1.000 & 1.000 \\
\hline
\end{tabular}
\end{center}

Both versions demonstrate consistent suppression of composite inputs, though the one-dimensional version yields slightly lower values for primes due to reduced smoothing freedom.

\end{document}